\documentclass[a4paper,11pt]{article}
\usepackage[margin=0.8in]{geometry}
\usepackage{enumitem}

\usepackage{braket}
\usepackage[backref=page,hidelinks]{hyperref}
\usepackage{amssymb,mathtools,amsthm,amsmath}
\usepackage{tikz}

\usepackage{xargs}

\usepackage{verbatim}

\theoremstyle{plain}
\newtheorem{theorem}{\bf Theorem}

\newtheorem{conjecture}[theorem]{\bf Conjecture}
\newtheorem{problem}[theorem]{\bf Problem}
\newtheorem{proposition}[theorem]{\bf Proposition}
\newtheorem{corollary}[theorem]{\bf Corollary}
\newtheorem{lemma}[theorem]{\bf Lemma}
\theoremstyle{definition}

\newenvironment{remark}[1][Remark.]{\begin{trivlist}
		\item[\hskip \labelsep {\bfseries #1}]}{\end{trivlist}}

\numberwithin{theorem}{section}
\numberwithin{equation}{section}

\DeclareMathOperator{\sa}{\text{sa}}
\DeclareMathOperator{\arbor}{\text{a}}
\DeclareMathOperator{\parden}{\tilde{\rho}}
\DeclareMathOperator{\den}{\rho}

\newcommand{\lfrac}[2]{\left\lfloor \frac{#1}{#2}\right\rfloor}

\newcommand{\ceilfrac}[2]{\left\lceil \frac{#1}{#2}\right\rceil}

\DeclareMathOperator{\eps}{\varepsilon}

\usepackage[d]{esvect}

\usepackage[nobysame,msc-links,non-sorted-cites,abbrev]{amsrefs}

\renewcommand{\eprint}[1]{\href{https://arxiv.org/abs/#1}{arXiv:#1}}

\BibSpec{book}{%
	+{}  {\PrintPrimary}                {transition}
	+{,} { \textbf}                     {title} 
	+{.} { }                            {part}
	+{:} { \textit}                     {subtitle}
	+{,} { \PrintEdition}               {edition}
	+{}  { \PrintEditorsB}              {editor}
	+{,} { \PrintTranslatorsC}          {translator}
	+{,} { \PrintContributions}         {contribution}
	+{,} { }                            {series}
	+{,} { \voltext}                    {volume}
	+{,} { }                            {publisher}
	+{,} { }                            {organization}
	+{,} { }                            {address}
	+{,} { \PrintDateB}                 {date}
	+{,} { }                            {status}
	+{}  { \parenthesize}               {language}
	+{}  { \PrintTranslation}           {translation}
	+{;} { \PrintReprint}               {reprint}
	+{.} { }                            {note}
	+{.} {}                             {transition}
	+{}  {\SentenceSpace \PrintReviews} {review}
}
\BibSpec{incollection}{%
  +{}  {\PrintAuthors}                {author}
  +{,} { \textit}                     {title}
  +{.} { }                            {part}
  +{:} { \textit}                     {subtitle}
  +{,} { \PrintContributions}         {contribution}
  +{,} { \PrintConference}            {conference}
  +{}  {\PrintBook}                   {book}
  +{,} { }                            {booktitle}
	+{}  { \PrintEditorsB}              {editor}
	+{,} { }                            {publisher}
  +{,} { \PrintDateB}                 {date}
  +{,} { pp.~}                        {pages}
  +{,} { }                            {status}
  +{,} { \PrintDOI}                   {doi}
  +{,} { available at \eprint}        {eprint}
  +{}  { \parenthesize}               {language}
  +{}  { \PrintTranslation}           {translation}
  +{;} { \PrintReprint}               {reprint}
  +{.} { }                            {note}
  +{.} {}                             {transition}
  +{}  {\SentenceSpace \PrintReviews} {review}
}

\makeatletter
\renewcommand{\PrintNames@a}[4]{%
	\PrintSeries{\name}
	{#1}
	{}{ and \set@othername}
	{,}{ \set@othername}
	{}{ and \set@othername}
	{#2}{#4}{#3}%
}
\makeatother

\usepackage{authblk}

\begin{document}

\title{Partition density, star arboricity, and sums of Laplacian eigenvalues of graphs}

    \author{Alan Lew\thanks{\href{mailto:alanlew@andrew.cmu.edu}{alanlew@andrew.cmu.edu}.}}
    \affil{Dept. Math. Sciences, Carnegie Mellon University, Pittsburgh, PA 15213, USA}

	\date{}
	\maketitle

\begin{abstract}
Let $G=(V,E)$ be a graph on $n$ vertices, and let $\lambda_1(L(G))\ge \cdots\ge \lambda_{n-1}(L(G))\ge \lambda_n(L(G))=0$ be the eigenvalues of its Laplacian matrix $L(G)$. 
Brouwer conjectured that for every $1\le k\le n$, $\sum_{i=1}^k \lambda_i(L(G)) \le |E|+\binom{k+1}{2}$.
Here, we prove the following weak version of Brouwer's conjecture: 
For every $1\leq k \leq n$,
\[
    \sum_{i=1}^k \lambda_i(L(G)) \leq 
       |E|+k^2+15k\log{k}+65k.
\]

For a graph $G=(V,E)$, we define its partition density $\tilde{\rho}(G)$ as the maximum, over all subgraphs $H$ of $G$, of the ratio between the number of edges of $H$ and the number of vertices in the largest connected component of $H$. Our argument relies on the study of the structure of the graphs $G$ satisfying $\tilde{\rho}(G)< k$.  In particular, using a result of Alon, McDiarmid and Reed, we show that every such graph can be decomposed into at most $k+ 15\log{k}+65$ edge-disjoint star forests (that is, forests whose connected components are all isomorphic to stars).

In addition, we show that for every graph $G=(V,E)$ and every $1\le k\le |V|$,
\[
    \sum_{i=1}^k \lambda_i(L(G)) \leq 
        |E|+k\cdot \nu(G) + \left\lfloor\frac{k}{2}\right\rfloor,
\]
where $\nu(G)$ is the maximum size of a matching in $G$.
\end{abstract}

\section{Introduction}

Let $G=(V,E)$ be a graph. For a vertex $v\in V$, the \emph{degree} of $v$ in $G$,  denoted by $\deg(v)$, is the number of edges of $G$ incident to $v$. The \emph{Laplacian matrix} of $G$, denoted by $L(G)$, is the $|V|\times |V|$ matrix defined by
\[
     L(G)_{u,v}=\begin{cases}
         \deg(u) & \text{if } u=v,\\
         -1 & \text{if } \{u,v\}\in E,\\
         0 & \text{otherwise,}
     \end{cases}
\]
for all $u,v\in V$. For $1\leq i\leq |V|$, we denote the $i$-th largest eigenvalue of $L(G)$ by $\lambda_i(L(G))$.

Brouwer proposed the following conjecture about the sum of the $k$ largest Laplacian eigenvalues of a graph $G$.

\begin{conjecture}[Brouwer {\cite[Section 3.1.1]{brouwer2012book}}]
\label{con:brouwer}
    Let $G=(V,E)$ be a graph. Then, for all $1\le k\le |V|$, 
    \begin{equation}\label{eq:brouwer}
    \sum_{i=1}^k \lambda_i(L(G))\le |E|+\binom{k+1}{2}.
    \end{equation}
\end{conjecture}

Brouwer's Conjecture can be seen as a variant of the Grone-Merris conjecture \cite{grone1994laplacian}, which was proved by Bai in \cite{bai2011gronemerris}.

\begin{theorem}[Bai \cite{bai2011gronemerris}]\label{thm:bai}
    Let $G=(V,E)$ be a graph. Then, for all  $1\le k\le |V|$,
    \[
        \sum_{i=1}^k \lambda_i(L(G))\le \sum_{i=1}^k \left|\left\{ v\in V:\, \deg(v)\ge i         \right\}\right|.
    \]
\end{theorem}

Conjecture \ref{con:brouwer} has received considerable attention in recent years, and some progress has been made. The $k=1$ case of the conjecture follows from the bound $\lambda_1(L(G))\le n'$, where $n'$ is the number of vertices in the largest connected component of $G$ (see Corollary \ref{cor:lambda_1_le_n}). The $k=|V|$ and $k=|V|-1$ cases are trivial, since $\sum_{i=1}^{|V|-1}\lambda_i(L(G))=\sum_{i=1}^{|V|}\lambda_i(L(G))=2|E|$. Haemers, Mohammadian and Tayfeh-Rezaie proved the $k=2$ case in \cite{haemers2010onthesum}.  
Li and Guo \cite{li2022full} proposed a conjectural characterization of the extremal cases of  \eqref{eq:brouwer}, and proved it for $k=2$ (this special case was previously conjectured by Guan, Zhai and Wu in \cite{guan2014onthesum}). 
Furthermore, Conjecture \ref{con:brouwer} has been verified for several families of graphs, such as trees \cite{haemers2010onthesum} (see \cite{fritscher2011onthesum} for an improved bound in this case), split graphs \cite{mayard2010thesis,berndsen2010thesis} and regular graphs \cite{mayard2010thesis,berndsen2010thesis}. Rocha showed in \cite{rocha2020aas} that the inequality in \eqref{eq:brouwer} holds asymptotically almost surely for various families of random graphs. Recently, in \cite{cooper2021constraints}, Cooper verified the conjecture for graphs whose degree sequences have small variance, generalizing previous results on regular graphs and on random graphs.
For more related work, see, for example, \cite{du2012upper,wang2012onaconjecture,das2015energy,chen2018note,chen2019onbrouwers,ganie2020further,cooper2021constraints} and the references therein.

Our main result is the following weak version of Conjecture \ref{con:brouwer}. 

\begin{theorem}\label{thm:weak_brouwer_improved}
Let $G=(V,E)$ be a graph. Then, for all $1\le k\le |V|$,
\[
    \sum_{i=1}^k \lambda_i(L(G))\le |E|+k^2+15k\log{k}+65k.
\]
\end{theorem}
Here and throughout the paper, $\log{x}$ denotes the natural logarithm of $x$.

Let $G=(V,E)$ be a graph. We write $E(G)=E$. For $U\subset V$, let $G[U]$ be the subgraph of $G$ induced by $U$. Let $n'(G)$ be the number of vertices in the largest connected component of $G$. We define
\begin{equation}\label{eq:parden1}
    \parden(G)= \max_{H\subset G} \frac{|E(H)|}{n'(H)},
\end{equation}
where the maximum is taken over all subgraphs $H$ of $G$ with at least one vertex. Since this maximum is always obtained by a spanning subgraph whose connected components are all induced subgraphs of $G$, we can write
\begin{equation}\label{eq:parden2}
\parden(G)=\max\left\{\frac{\sum_{i=1}^m |E(G[V_i])|}{\max_{1\le i\le m} |V_i|} 
  :\, m\ge 1,\, V_1,\ldots,V_m \text{ is a partition of } V \right\}.
\end{equation}
Hence, we call $\parden(G)$ the \emph{partition density} of $G$. 
The partition density of a graph is tightly related to various classical properties, such as the arboricity, $k$-orientability (see Section \ref{sec:prelims_orientability}),   and the matching number of the graph (see Section \ref{sec:prelims_matching}). 

A \emph{star forest} is a graph whose connected components are all isomorphic to star graphs.
The \emph{star arboricity} of $G$, denoted by $\sa(G)$, is the minimum $t$ such that $G$ can be decomposed into $t$ edge-disjoint star forests. The notion of star arboricity was introduced by Akiyama and Kano in \cite{akiyama1982path}, and further studied in \cite{algor1989thestar,aoki1990star,kurek1992arboricity,alon1992star,hakimi1996star} (see Section \ref{sec:prelims_orientability} for more details on star arboricity and its relation to other graph parameters). 

The main ingredient in the proof of Theorem \ref{thm:weak_brouwer_improved} is the following result about the star arboricity of graphs with low partition density.

\begin{theorem}
    \label{thm:sa}
    Let $k\ge 1$, and let $G=(V,E)$ be a graph with $\parden(G)<k$. Then, 
    \[
    \sa(G)\le k+15\log{k}+65.
    \]
\end{theorem}

The proof of Theorem \ref{thm:sa} relies on an argument by Alon, McDiarmid and Reed from \cite{alon1992star} (see Section \ref{sec:structure} for more details). The proof of Theorem \ref{thm:weak_brouwer_improved} may be outlined as follows: First, we reduce the problem to the case of graphs with partition density smaller than $k$ (Lemma \ref{lemma:sparse_subgraph}). Then, we show that the sum of the $k$ largest Laplacian eigenvalues of a graph $G=(V,E)$ is at most $|E|+k\cdot \sa(G)$ (Lemma \ref{lemma:star_arboricity_bound}). Finally, the claim follows from Theorem \ref{thm:sa}.

Next, we study the relation between the sum of the $k$ largest Laplacian eigenvalues and the matching number of a graph.
A \emph{matching} in a graph $G$ is a set of pairwise disjoint edges. The \emph{matching number} of $G$, denoted by $\nu(G)$, is the maximum size of a matching in $G$. A \emph{vertex cover} of a graph $G=(V,E)$ is a set $C\subset V$ such that $C\cap e\ne \emptyset$ for all $e\in E$.
The \emph{covering number} of $G$, denoted by $\tau(G)$, is the minimum size of a vertex cover of $G$.
Das, Mojallal and Gutman showed in \cite{das2015energy} that for every graph $G$ and every $1\leq k\leq |V|$, $\sum_{i=1}^k \lambda_i(L(G)) \le |E|+k\cdot \tau(G)$ (see Corollary \ref{cor:covering}). Since, for every graph $G$,  $\tau(G)\le 2 \nu(G)$, we obtain, as a consequence,
\begin{equation}\label{eq:old_bound}
\sum_{i=1}^k \lambda_i(L(G)) \le |E|+2k\cdot \nu(G).
\end{equation}

Here, we prove the following inequality, which improves on the bound in \eqref{eq:old_bound}.

\begin{theorem}\label{thm:matching_number}
Let $G=(V,E)$ be a graph. Then, for all $1\leq k\leq |V|$,
\[
    \sum_{i=1}^k \lambda_i(L(G)) \le |E|+ k\cdot \nu(G) +\lfrac{k}{2}.
\]
\end{theorem}

As a consequence of Theorem \ref{thm:matching_number}, we obtain the following result.

\begin{corollary}\label{cor:weak_brouwer}
Let $G=(V,E)$ be a graph. Then, for all $1\le k\le |V|$, 
    \[
    \sum_{i=1}^k \lambda_i(L(G))\le |E|+2k^2-\ceilfrac{k}{2}.
    \]
\end{corollary}

Note that Corollary \ref{cor:weak_brouwer} is weaker than Theorem \ref{thm:weak_brouwer_improved} for large values of $k$, but improves upon it for $k\le 139$.

The paper is organized as follows. In Section \ref{sec:prel}, we introduce some background material on Laplacian eigenvalues, orientability, arboricity, star arboricity, and matching theory, which we will later use. In Section \ref{sec:structure}, we study the structure of graphs with low partition density, and we prove Theorem \ref{thm:sa} and our main result, Theorem \ref{thm:weak_brouwer_improved}.
In Section \ref{sec:matchings}, we present the proofs of Theorem \ref{thm:matching_number} and Corollary \ref{cor:weak_brouwer}. Finally, in Section \ref{sec:conc}, we present some related conjectures and discuss possible directions for future research.

\section{Preliminaries}\label{sec:prel}

\subsection{Basic eigenvalue bounds}

For a graph $G=(V,E)$ and  $1\le k \le |V|$, we define
\[
    \eps_k(G)= \left(\sum_{i=1}^k \lambda_i(L(G))\right)-|E|.
\]
For convenience, for $k>|V|$, we define $\eps_k(G)=\left(\sum_{i=1}^{|V|}\lambda_i(L(G))\right)- |E|=|E|$. 

We will need the following basic facts about Laplacian eigenvalues of graphs.

\begin{lemma}[See {\cite[Proposition 1.3.6]{brouwer2012book}}]\label{lemma:components}
Let $G$ be a graph, and let $G_1,\ldots,G_t$ be its connected components. Then, the spectrum of $L(G)$ is the union of the spectra of $L(G_i)$, for $1\le i\le t$ (in particular, the multiplicity of every eigenvalue $\lambda$ of $L(G)$ is the sum of its multiplicities as an eigenvalue of $L(G_i)$, for $1\le i\le t$).
\end{lemma}

\begin{lemma}[Anderson, Morley \cite{anderson1985eigenvalues}; see also {\cite[Section 1.3.2]{brouwer2012book}}]\label{lemma:n_bound}
    Let $G=(V,E)$ be a graph. Then,
    \[
        \lambda_1(L(G))\le |V|.
    \]
\end{lemma}

Recall that, for a graph $G=(V,E)$, we denote by $n'(G)$ the number of vertices in the largest connected component of $G$.

\begin{corollary}\label{cor:lambda_1_le_n}
    Let $G=(V,E)$ be a graph. Then,
    \[
    \lambda_1(L(G))\le n'(G).
    \]
\end{corollary}
\begin{proof}
    Let $G_1=(V_1,E_1),\ldots,G_t=(V_t,E_t)$ be the connected components of $G$. By Lemma \ref{lemma:components},
    \[
    \lambda_1(L(G))= \max_{1\le i\le t} \lambda_1(L(G_i)).
    \]
    By Lemma \ref{lemma:n_bound}, $\lambda_1(L(G_i))\le |V_i|\le n'(G)$ for all $1\le i\le t$. So, $\lambda_1(L(G))\le n'(G)$, as wanted.
\end{proof}

\begin{corollary}\label{cor:n'_bound}
 Let $G=(V,E)$ be a graph. Then, for $1\le k\le |V|$,
 \[
    \eps_k(G)\le k\cdot n'(G)-|E|.
 \]
\end{corollary}
\begin{proof}
    By Corollary \ref{cor:lambda_1_le_n},
    $
\eps_k(G)=\sum_{i=1}^k \lambda_i(L(G)) - |E| \le k\cdot \lambda_1(L(G))-|E|\le k\cdot n'(G)-|E|.
    $
\end{proof}

The following two results are implicit in \cite{haemers2010onthesum}. For completeness, we include here their proofs.

We say that the graphs $G_1=(V_1,E_1),\ldots,G_t=(V_t,E_t)$ are \emph{edge-disjoint} if $E_i\cap E_j=\emptyset$ for all $1\le i<j\le t$. The \emph{union} of $G_1,\ldots,G_t$ is the graph $G=(V,E)$,  where $V=\cup_{i=1}^t V_i$ and $E=\cup_{i=1}^t E_i$.

\begin{lemma}[See {\cite[Corollary 1]{haemers2010onthesum}}]
\label{lemma:excess}
    Let $G_1,\ldots, G_t$ be edge-disjoint graphs, and let $G=(V,E)$ be their union.
    Then, for all $1\leq k\leq |V|$, 
    \[
        \eps_k(G)\le \sum_{i=1}^t \eps_k(G_i).
    \]
\end{lemma}
\begin{proof}

    Since adding isolated vertices does not affect $\eps_k(G_i)$, we may assume all graphs $G_i$ have the same vertex set $V$. For $1\le i\le t$, let $E_i$ be the edge set of $G_i$. Then, we have $G=(V,E)$, where  $E=\cup_{i=1}^t E_i$.
    Note that $L(G)=\sum_{i=1}^t L(G_i)$. Hence, Ky Fan's inequality (\cite{fan1949onatheorem}, see also \cite[Proposition A.6]{marshall2011inequalities}),
    \[
        \sum_{j=1}^k \lambda_j(L(G)) \le \sum_{i=1}^t \sum_{j=1}^k \lambda_j(L(G_i)).
    \]
    Therefore, using the fact that $|E|=\sum_{i=1}^t |E_i|$, we obtain
    \[
        \eps_k(G) =   \sum_{j=1}^k \lambda_j(L(G))- |E| \le \sum_{i=1}^t \left( \left(\sum_{j=1}^k \lambda_j(L(G_i))\right)- |E_i|\right) = \sum_{i=1}^t \eps_k(G_i).
    \]
\end{proof}

We say that a graph $G=(V,E)$ is \emph{non-empty} if $|E|>0$. The following result is a simple consequence of Lemma \ref{lemma:excess}.

\begin{lemma}[See {\cite[Lemma 5]{haemers2010onthesum}}, {\cite[Corollary 2.5]{wang2012onaconjecture}}]
\label{lem:removable_graphs}
    Let $G=(V,E)$ be a graph, and let $1\leq k\leq |V|$. Then, there exists a subgraph $G'$ of $G$ such that $\eps_k(G')\ge \eps_k(G)$, and every non-empty subgraph $H$ of $G'$ satisfies $\eps_k(H)>0$.
\end{lemma}
\begin{proof}
    We argue by induction on the size of $E$. If $|E|=0$, the claim holds trivially.  
    Assume $|E|>0$, and assume that the claim holds for all graphs with less than $|E|$ edges. If every non-empty subgraph $H$ of $G$ satisfies $\eps_k(H)>0$, we may take $G'=G$, and we are done.

    Otherwise, there is a non-empty subgraph $H$ of $G$ such that $\eps_k(H)\le 0$. Let $\tilde{G}$ be the graph obtained by removing from $G$ the edges of $H$. By Lemma \ref{lemma:excess}, we have
    \[
        \eps_k(G)\le \eps_k(\tilde{G})+\eps_k(H)\le \eps_k(\tilde{G}).
    \]
    Since the number of edges of $\tilde{G}$ is strictly smaller than $|E|$, by the induction hypothesis there exists a subgraph $G'$ of $\tilde{G}$ such that $\eps_k(H')>0$ for all non-empty subgraphs $H'$ of $G'$, and
    \[
        \eps_k(G')\ge \eps_k(\tilde{G})\ge \eps_k(G).
    \]
    Since $G'$ is also a subgraph of $G$, we are done.
\end{proof}

As an application of Lemma \ref{lem:removable_graphs}, we obtain the following key result.

\begin{lemma}\label{lemma:sparse_subgraph}
    Let $G=(V,E)$ be a graph, and let $1\le k\le |V|$. Then, there exists a subgraph $G'$ of $G$ such that $\parden(G')< k$ and $\eps_k(G')\ge \eps_k(G)$.
\end{lemma}
\begin{proof}
By Lemma \ref{lem:removable_graphs}, there exists a subgraph $G'$ of $G$ with $\eps_k(G')\ge \eps_k(G)$ such that $\eps_k(H)>0$ for every non-empty subgraph $H$ of $G'$. Assume for contradiction that $\parden(G')\ge k$. Then, by \eqref{eq:parden1},  there exists a subgraph $H$ of $G'$ with at least one vertex for which $|E(H)|\ge k\cdot n'(H)$. Clearly, $H$ is non-empty. By Corollary \ref{cor:n'_bound}, we have
\[
    \eps_k(H) \le k \cdot n'(H)-|E(H)| \le 0,
\]
a contradiction. Therefore, $\parden(G')< k$.
\end{proof}

\subsection{Orientability, arboricity, and star arboricity}
\label{sec:prelims_orientability}

Let $G=(V,E)$ be a graph, and let $h:E\to V$ such that $h(e)\in e$ for all $e\in E$. The \emph{orientation} of $G$ associated to $h$ is the directed graph $D=(V,A)$, where
\[
    A= \left\{ \vv{uv}:\, \{u,v\}\in E, \, h(\{u,v\})=v\right\}.
\]
The \emph{in-degree} of a vertex $v$ in $D$ is the number of directed edges of the form $\vv{uv}$, for some $u\in V$. If the in-degree of every vertex is at most $k$, we say that $D$ is a \emph{$k$-orientation} of $G$.  We say that $G$ is \emph{$k$-orientable} if there exists a $k$-orientation of $G$.\footnote{The notion of $k$-orientability is often defined in terms of the existence of an orientation with out-degrees bounded by $k$ (instead of in-degrees). These two definitions are clearly equivalent; we chose to define $k$-orientability in terms of the in-degrees in order to be consistent with the terminology in \cite{alon1992star}.}

We define the \emph{density} of a graph $G=(V,E)$ as
\[
    \den(G)= \max \left\{ \frac{|E(G[U])|}{|U|} :\, \emptyset\ne U\subset V\right\}.
\]
Clearly, $\den(G)\le \parden(G)$ for every graph $G$. The following theorem states that the $k$-orientable graphs are exactly the graphs with density at most $k$.

\begin{theorem}[Hakimi {\cite[Theorem 4]{hakimi1965onthedegrees}}, Frank, Gy\'arf\'as {\cite[Theorem 1]{frank1978orient}}]\label{thm:frank}
    Let $G=(V,E)$ be a graph. Then, $G$ is $k$-orientable if and only if $\den(G)\le k$.
\end{theorem}

The \emph{arboricity} of a graph $G=(V,E)$, denoted by $\arbor(G)$, is the minimum integer $t$ such that $G$ can be decomposed into $t$ edge-disjoint forests. Nash-Williams showed in \cite{nash-williams1964decomposition} that 
\begin{equation}\label{eq:nash_williams}
    a(G)= \max \left\{ \left\lceil\frac{|E(G[U])|}{|U|-1}\right\rceil :\, U\subset V,\, |U|\ge 2 \right\}.
\end{equation}
Since, for every $U\subset V$ of size at least $2$, 
\[
\frac{|E(G[U])|}{|U|-1}-\frac{|E(G[U])|}{|U|} =\frac{|E(G[U])|}{|U|(|U|-1)}\le \frac{1}{2},
\]
we obtain, by \eqref{eq:nash_williams},
\begin{equation}
    \label{eq:rho_arb}
    a(G)\le \left\lceil \rho(G)+1/2 \right\rceil.
\end{equation}

Recall that a star forest is a graph whose connected components are all isomorphic to star graphs, and the star arboricity of a graph $G$, denoted by $\sa(G)$, is the minimum $t$ such that $G$ can be decomposed into $t$ edge-disjoint star forests.  
Clearly, $\arbor(G)\le \sa(G)$. On the other hand, it was show by Algor and Alon in \cite[Lemma 4.1]{algor1989thestar} that, for every graph $G$,
\begin{equation}\label{eq:sa_at_most_twice_a}
    \sa(G)\le 2\cdot \arbor(G).
\end{equation}
Alon, McDiarmid and Reed \cite{alon1992star}, and independently Kurek \cite{kurek1992arboricity}, presented examples showing that this bound is sharp.

The following lemma gives an upper bound on $\eps_k(G)$ in terms of the star arboricity of $G$.

\begin{lemma}\label{lemma:star_arboricity_bound}
    Let $G=(V,E)$ be a graph, and let $1\le k\le |V|$. Then,
    \[
    \eps_k(G)\le k\cdot \sa(G).
    \]
\end{lemma}
\begin{proof}
    Let $n\ge 2$, and let $S_n$ be the star graph with $n$ vertices. It is well-known and easy to check (see, for example, \cite{anderson1985eigenvalues} or \cite[Lemma 2]{haemers2010onthesum}) that the Laplacian eigenvalues of $S_n$ are $n$ with multiplicity $1$, $1$ with multiplicity $n-2$, and $0$ with multiplicity $1$. 

    Let $F$ be a star forest with connected components isomorphic to $S_{n_1},\ldots,S_{n_t}$, respectively, for some $n_1\ge n_2 \ge \ldots\ge n_t$. Since removing isolated vertices does not affect $\eps_k(F)$, we may assume that $n_i\ge 2$ for all $1\le i\le t$. Then, by Lemma \ref{lemma:components}, the Laplacian eigenvalues of $F$ are $n_1,\ldots,n_t$, each repeated once, $1$ with multiplicity $n_1+\cdots+n_t-2t$, and $0$ with multiplicity $t$. 
    Hence, for $k\ge 1$,
    \[
        \sum_{i=1}^k \lambda_i(L(F))=\begin{cases}
    n_1+\cdots+n_k   & \text{if } k\le t,\\
    n_1+\cdots+n_t+(k-t) & \text{if } t<k\le n_1+\cdots+n_t-t,\\
    2(n_1+\cdots+n_t)-2t & \text{otherwise.}
        \end{cases}
    \]
    In all cases, we have
    \[  
        \sum_{i=1}^k \lambda_i(L(F)) \le n_1+\cdots+n_t+(k-t)
        = |E(F)|+k.
    \]
    Thus, $\eps_k(F)\le k$ for all $k\ge 1$. 
     Finally, let $m=\sa(G)$, and let $F_1,\ldots,F_m$ be a decomposition of $G$ into edge-disjoints star forests. Then, by Lemma \ref{lemma:excess},
     \[
        \eps_k(G)\le \sum_{i=1}^m \eps_k(F_i)\le k\cdot m= k\cdot \sa(G).
     \]
    
\end{proof}

As a consequence, we obtain the following result, which was proved by Das, Mojallal and Gutman in \cite{das2015energy}. 

\begin{corollary}[Das, Mojallal, Gutman \cite{das2015energy}]
\label{cor:covering}
    Let $G=(V,E)$ be a graph, and let $1\le k\le |V|$. Then,
    \[
    \eps_k(G)\le k\cdot \tau(G).
    \]
\end{corollary}

Corollary \ref{cor:covering} follows immediately from Lemma \ref{lemma:star_arboricity_bound} and the following simple lemma. The argument below is essentially the same as the one in \cite{das2015energy}, but we include it here for completeness.

\begin{lemma}\label{lemma:sa_at_most_tau}
    Let $G=(V,E)$ be a graph. Then, $\sa(G)\le \tau(G)$.
\end{lemma}
\begin{proof}
 Let $C=\{v_1,\ldots,v_t\}$ be a vertex cover of $G$ of size $t=\tau(G)$. For $1\leq i\leq t$, let
    \[
        E_i= \left\{e\in E:\, v_i\in e,\, v_j\notin e \text{ for all } 1\leq j<i\right\},
    \]
    and let $G_i=(V,E_i)$. Note that each of the graphs $G_i$ is a star forest (since each $G_i$ is isomorphic to a star graph, plus, possibly, some isolated vertices). Moreover, the graphs $G_1,\ldots,G_t$ are edge-disjoint, and $G=\cup_{i=1}^t G_i$. Therefore, $\sa(G)\le t= \tau(G)$.
\end{proof}

\begin{remark}
    A result analogous to Lemma \ref{lemma:star_arboricity_bound}, \[
    \eps_k(G)\le (2k-1)\cdot \arbor(G),
    \]
    was observed by Cooper in \cite{cooper2021constraints}. The bound follows from a result of Haemers, Mohammadian and Tayfeh-Rezaie (\cite[Theorem 5]{haemers2010onthesum}), which states that $\eps_k(T)\le 2k-1$ for every tree $T$.
\end{remark}

\subsection{Matching theory}\label{sec:prelims_matching}

Let $\ell\ge 1$ be an integer. Recall that $S_{\ell+1}$ is the star graph with $\ell+1$ vertices (and $\ell$ edges). We say that a vertex $v$ is a \emph{center} of $S_{\ell+1}$ if its degree in $S_{\ell+1}$ is $\ell$.
An \emph{$\ell$-star forest} is a graph $F=(V,E)$ whose connected components are all isomorphic to $S_{\ell+1}$. For $U\subset V$, we say that $F$ is \emph{centered} at $U$ if each connected component of $F$ has a center in $U$.

For a graph $G=(V,E)$, let $\nu_{\ell}(G)$ be the maximum number of connected components in an $\ell$-star forest contained in $G$. For $U\subset V$, let $\nu_{\ell}(G;U)$ be the maximum number of connected components in an $\ell$-star forest centered at $U$ and contained in $G$.
Since a $1$-star forest is just a matching, we have $\nu_1(G)=\nu(G)$.

\begin{lemma}\label{lemma:parden_matching}
 Let $G=(V,E)$ be a graph, and let $\ell\ge 1$. Then,
 \[
    \nu_{\ell}(G)\le\left\lfloor \left(1+\frac{1}{\ell}\right)\parden(G)   \right\rfloor. 
\]  
\end{lemma}
\begin{proof}
Let $F$ be an $\ell$-star forest contained in $G$ with $\nu_{\ell}(G)$ connected components. By \eqref{eq:parden1}, 
\[
    \parden(G) \ge \frac{|E(F)|}{n'(F)} =\frac{\ell \cdot \nu_{\ell}(G)}{\ell+1}.
\]
So,
\[
\nu_{\ell}(G) \le \left(1+\frac{1}{\ell}\right)\parden(G).
\]
\end{proof}

\begin{corollary}\label{corollary:bounded_generalized_matching}
    Let $G=(V,E)$ be a graph. Let $1\le k\le |V|$, and let $\ell\ge 1$. Then, there exists a subgraph $G'$ of $G$ with $\nu_{\ell}(G')<\left(1+1/\ell\right)k$ such that $\eps_k(G')\ge \eps_k(G)$.
\end{corollary}
\begin{proof}
    By Lemma \ref{lemma:sparse_subgraph}, there is a subgraph $G'$ of $G$ with $\eps_k(G')\ge \eps_k(G)$ and $\parden(G')<k$. By Lemma \ref{lemma:parden_matching}, we have
    \[
        \nu_{\ell}(G')\le \left(1+\frac{1}{\ell}\right)\parden(G') <\left(1+\frac{1}{\ell}\right)k.
    \]
\end{proof}
Let us note that the special cases $k=2$, $\ell=1, 2$ of Corollary \ref{corollary:bounded_generalized_matching} are implicit in \cite{haemers2010onthesum}.

For a graph $G=(V,E)$ and $V'\subset V$, let $N_G(V')$ be the \emph{neighborhood} of $V'$ in $G$ (that is, the set of vertices in $V\setminus V'$ having at least one neighbor in $V'$).

The following result can be obtained from Hall's Theorem by a standard vertex duplication argument.

\begin{proposition}[Halmos, Vaughan \cite{halmos1950marriage}; see also {\cite[Theorem 4]{mirsky1969hall}}]
\label{prop:hall}
    Let $G=(A,B,E)$ be a bipartite graph, and let $\ell\ge 1$ be an integer. Then, $\nu_{\ell}(G;A)=|A|$ if and only if $|N_G(A')|\ge \ell \cdot |A'|$ for all $A'\subset A$.
\end{proposition}

We will need the following consequence of Proposition \ref{prop:hall}.

\begin{corollary}\label{cor:hall}
    Let $G=(A,B,E)$ be a bipartite graph, and let $\ell \ge 1$. Assume $\nu_{\ell}(G;A)<|A|$. Then,  there exists $\emptyset\ne A'\subset A$ such that $|A\setminus A'|\le \nu_{\ell}(G; A)$ and  $|N_G(A')|\le \ell\cdot |A'|-1$.
\end{corollary}
\begin{proof}
    Let
    \[
    \mathcal{B}=\{\emptyset\}\cup \left\{A''\subset A:\, |N_G(A'')|<\ell |A''|\right\}.
    \]
    Let $A'$ be a set of maximum size in $\mathcal{B}$. Assume for contradiction that $|A\setminus A'|>\nu_{\ell}(G;A)$. 
    Let $G'=G[V\setminus A']$. Since $\nu_{\ell}(G';A\setminus A')\le \nu_{\ell}(G;A)<|A\setminus A'|$, by Proposition \ref{prop:hall} there exists $A''\subset A\setminus A'$ such that $|N_G(A'')|= |N_{G'}(A'')| < \ell |A''|$. Let $\tilde{A}=A'\cup A''$. Clearly, $A''\ne\emptyset$, and so $|\tilde{A}|>|A'|$. 
    But $|N_G(\tilde{A})|\le |N_G(A')|+|N_G(A'')|< \ell |A'|+\ell|A''|= \ell|\tilde{A}|$, in contradiction to the maximality of $A'$ in $\mathcal{B}$. Hence, $|A\setminus A'|\le \nu_{\ell}(G;A)$. Since $|A|>\nu_{\ell}(G;A)$, we must have $A'\ne \emptyset$. Therefore, since $A'\in\mathcal{B}$, $|N_G(A')|\le \ell|A'|-1$, as wanted.
\end{proof}

In addition, we will need the following characterization of the matching number of a graph. Let $G=(V,E)$ be a graph. Let $v_1,\ldots,v_r\in V$, and let $S_1,\ldots,S_t\subset V$ be pairwise disjoint sets of odd size. We say that the family $C=\{v_1,\ldots,v_r,S_1,\ldots,S_t\}$ is an \emph{odd set cover} of $G$ if every edge of $G$ is either incident to one of the vertices $v_1,\ldots,v_r$, or is contained in one of the odd sets $S_1,\ldots,S_t$. We define the \emph{weight} of the odd set cover $C$ as
\[
    w(C)= r+\sum_{i=1}^t \frac{|S_i|-1}{2}.
\]

Edmonds proved in \cite{edmonds1965paths} the following result,  which can be seen as an analogue of K\"onig's theorem for general graphs. 

\begin{theorem}[Edmonds {\cite[Section 5.6]{edmonds1965paths}}; see also {\cite[Exercise 3.1.17]{lovasz1986book}}]\label{thm:edmonds}
Let $G=(V,E)$ be a graph. Then,
\[
    \nu(G)=\min\{ w(C) :\, C \text{ is an odd set cover of } G \}.
\]
\end{theorem}

\begin{remark}
    Note that, although the definition of an odd set cover as stated in \cite{edmonds1965paths} and \cite{lovasz1986book} does not require the sets $S_1,\ldots,S_t$ in the cover to be pairwise disjoint, we may add that restriction without loss of generality. Indeed, let $C=\{v_1,\ldots,v_r,S_1,\ldots,S_t\}$ be an odd set cover of a graph $G$, and assume that two sets in the cover, say $S_1$ and $S_2$, have a non-empty intersection. Then, we may define a new odd set cover 
    \[
    C'=\begin{cases}
        (C \setminus \{S_1,S_2\})\cup \{ S_1\cup S_2\} & \text{if } |S_1\cup S_2| \text{ is odd,}\\
        (C\setminus \{S_1,S_2\})\cup\{S_1\cup S_2\setminus \{v\},v\} & \text{if } |S_1\cup S_2| \text{ is even,}
    \end{cases}
    \]
    where $v$ may be any vertex in $S_1\cup S_2$. It is easy to check that in both cases $C'$ is indeed an odd set cover of $G$, and that $w(C')\le w(C)$. By repeatedly performing this operation, we obtain an odd set cover with pairwise disjoint odd sets, whose weight is at most $w(C)$. 
\end{remark}

\section{The structure of graphs with low partition density}
\label{sec:structure}

In this section, we prove our main result, Theorem \ref{thm:weak_brouwer_improved}. First, we present the proof of Theorem \ref{thm:sa}, which gives an upper bound on the star arboricity of graphs with low partition density. 
We will need the following result about the structure of graphs $G$ with $\parden(G)<k$.

For a graph $G=(V,E)$ and $A,B\subset V$ such that $A\cap B=\emptyset$, we denote by $G[A,B]$ the subgraph of $G$ with vertex set $A\cup B$ and edge set $\{\{u,v\}\in E :\, u\in A,\, v\in B\}$. Note that $G[A,B]$ is a bipartite graph, with partite sets $A$ and $B$. We say that a set $U\subset V$ is an \emph{independent set} in $G$ if $G[I]$ has no edges.

\begin{proposition}\label{prop:structure} Let $k\ge 1$, and let $G=(V,E)$ be a graph with $\parden(G)<k$. Then, there exist pairwise disjoint sets $U,C,I\subset V$ such that $V=U\cup C \cup I$, $|U|\le 4k^2+2k-3$, $|C|\le k$, and $I$ is an independent set in $G$ with $N_G(I)\subset C$.
\end{proposition}
\begin{proof}
    By Lemma \ref{lemma:parden_matching}, we have $\nu(G)\le 2k-1$ and $\nu_k(G)\le k$.
        Let $S\subset V$ be a vertex cover of $G$ of size $|S|=\tau(G)\le 2\nu(G)\le 4k-2$. Since $S$ is a vertex cover, $V\setminus S$ is an independent set in $G$.

    If $|S|\le k$, we may choose $C=S$, $U=\emptyset$ and $I=V\setminus S$, and we are done.    
    So, we assume $|S|>k$. Look at the bipartite graph $G'=G[S,V\setminus S]$. Since $\nu_k(G';S)\le \nu_k(G)\le k<|S|$, by Corollary \ref{cor:hall} there is a set $\emptyset\ne S'\subset S$ such that $|S\setminus S'|\le k$, and 
    \[
    |N_{G'}(S')|\le k|S'|-1\le k(4k-2)-1 = 4k^2-2k-1.
    \]
    Let $C=S\setminus S'$, $U=S'\cup N_{G'}(S')$ and $I=V\setminus (C\cup U)$.
    Then, we have $|C|\le k$ and $|U|\le |S'|+|N_{G'}(S')|\le 4k-2 + (4k^2-2k-1)= 4k^2+2k-3$. Since $I\subset V\setminus S$,  $I$ is an independent set in $G$. 
    We are left to show that $N_G(I)\subset C$. Indeed, let $u\in I$ and $v\in V\setminus I$ such that $\{u,v\}\in E$. Since $S$ is a vertex cover, $v\in S$. If $v\in S'$, then $u\in N_{G'}(S')$, a contradiction to $I\cap U=\emptyset$. Therefore, $v\in S\setminus S'=C$, as wanted.
\end{proof}

For a graph $G$, let $\Delta(G)$ be the maximum degree of a vertex in $G$.
Alon, McDiarmid and Reed proved in \cite{alon1992star} that for every graph $G$, $\sa(G)\le k+15\log{k}+6\log{(\Delta(G))}/\log{k}+65$. Their argument relies on the following two key results.

We say that a family of (not necessarily distinct) sets $\{A_1,\ldots,A_m\}$ has a \emph{transversal} (or system of distinct representatives) if there exist $m$ distinct elements $x_1,\ldots,x_m$ such that $x_i\in A_i$ for all $1\le i\le m$.

Let $k,c\ge 1$ be integers, and let $D=(V,A)$ be a directed graph.  Let $\{L_v:\, v\in V\}$ be a family of subsets of $[k+c]=\{1,2,\ldots,k+c\}$ such that, for all $v\in V$, the set $L_v$ is of size at most $c$, and the family 
\[
    \mathcal{L}_v=\{L_u:\, \vv{uv}\in A\}
\]
has a transversal.
We call the family $\{L_v:\, v\in V\}$ a  \emph{$(k,c)$-assignment} of $D$.

\begin{lemma}[Alon, McDiarmid, Reed {\cite[Lemma 2.3]{alon1992star}}]\label{lem:alon_coloring}
    Let $k,c\ge 1$. Let $G=(V,E)$ be a graph, and let $D=(V,A)$ be a $k$-orientation of $G$. If $D$ has a $(k,c)$-assignment, then $\sa(G)\le k+3c+2$.
\end{lemma}

\begin{proposition}[Alon, McDiarmid, Reed \cite{alon1992star}\footnote{Proposition \ref{prop:alon_coloring_2} is implicit in the proof of \cite[Theorem 2.1]{alon1992star} (see page 378 of \cite{alon1992star}).}]\label{prop:alon_coloring_2}
   Let $k\ge 2$. Let $G=(V,E)$ be a graph, and let $D=(V,A)$ be a $k$-orientation of $G$. 
    Let $c$ be an integer satisfying $k\ge c\ge \max\{ 5\log{k}+20, 2\log{(\Delta(G))}/\log{k}+11\}$. Then,  $D$ has a $(k,c)$-assignment.
\end{proposition}

\begin{proof}[Proof of Theorem \ref{thm:sa}]
    Let $k\ge 1$, and let $G=(V,E)$ be a graph with $\parden(G)<k$.  
    If $k\le 100$, then, since $\den(G)\le \parden(G)<k$, we obtain, by \eqref{eq:sa_at_most_twice_a} and \eqref{eq:rho_arb},
    \[
        \sa(G)\le 2\cdot \arbor(G)\le 2\lceil\den(G)+1/2\rceil\le 2(k+1) \le k+15\log{k}+65,
    \]
    as wanted.  
    Therefore, we assume $k> 100$. Let $c=\lceil5\log{k}+20\rceil$. Note that, since $k> 100$, we have $c\le k$. 
    By Proposition \ref{prop:structure}, there exist pairwise disjoint sets $U, C, I\subset V$ such that $V=U\cup C\cup I$,  $|U|\le 4k^2+2k-3$, $|C|\le k$, and $I$ is an independent set in $G$ with $N_G(I)\subset C$. 

    Since $\den(G)\le \parden(G)<k$, by Theorem \ref{thm:frank} $G$ is $k$-orientable. Let $D=(V,A)$ be a $k$-orientation of $G$. Since $\deg(v)\le k$ for all $v\in I$, we may assume that all edges incident to a vertex $v\in I$ are oriented in $D$ towards $v$.

    Let $G'=(V',E')$ be the graph obtained from $G[U\cup C]$ by adding a new vertex $x$ and adding all the edges of the form $\{x,v\}$ for $v\in C$. Note that $G'$ is $k$-orientable. Indeed, we may first orient all the edges of $G[U\cup C]$ according to their orientation in $D$, and then, for all $v\in C$, orient the edge $\{v,x\}$ as $\vv{vx}$. Since the degree of $x$ in $G'$ is at most $k$, we obtain a $k$-orientation of $G'$. We denote this orientation by $D'=(V',A')$.
    
    We have $\Delta(G')\le |V'|-1=|U|+|C|\le 5k^2$, and therefore
    \[
    \frac{2\log{(\Delta(G'))}}{\log{k}}+11 \le \frac{4\log{k}+2\log{5}}{\log{k}}+11< 20 \le c.
    \]
    Hence, by Proposition \ref{prop:alon_coloring_2}, $D'$ has a $(k,c)$-assignment $\{L'_v:\, v\in V'\}$. So, $|L_v'|\le c$, and the family $\mathcal{L}'_v=\{L_u':\, \vv{uv}\in A'\}$ has a transversal, for all $v\in V'$. 
    We may extend this assignment to a $(k,c)$-assignment of $D$, in the following way. For $v\in V$, let $L_v=L'_v$ if $v\in U\cup C$, and  $L_v=\emptyset$ if $v\in I$. 
     Clearly, $|L_v|\le c$ for all $v\in V$.   We are left to show that for every $v\in V$, the family
    $
        \mathcal{L}_v= \{L_u:\, \vv{uv}\in A\}
    $     has a transversal. Indeed, let $v\in V$. Then, since all the edges in $D$ incident to a vertex in $I$ are oriented towards that vertex, we have $\mathcal{L}_v=\mathcal{L}'_v$ if $v\in U\cup C$, and, since $N_G(I)\subset C$, $\mathcal{L}_v\subset \mathcal{L}'_x$ if $v\in I$. Therefore, in both cases, $\mathcal{L}_v$ has a transversal.
    Hence, $\{L_v:\, v\in V\}$ is a $(k,c)$-assignment of $D$.
    
    By Lemma \ref{lem:alon_coloring}, we obtain
    \[
        \sa(G)\le k+3c+2 \le k+15\log{k}+65.
    \]
\end{proof}

The next result shows that the bound in Theorem \ref{thm:sa} is optimal, up to lower order terms.

\begin{proposition}
 Let $k\ge 1$, and let $n\ge k^2-2k+2$.  Let $K_{k,n}$ be the complete bipartite graph with parts of size $k$ and $n$ respectively. Then, $\sa(K_{k,n})=k$ and
 $
 \parden(K_{k,n})<k$.
\end{proposition}
\begin{proof}
   By Lemma \ref{lemma:sa_at_most_tau}, $\sa(K_{k,n})\le \tau(K_{k,n})=k$. On the other hand, by \eqref{eq:nash_williams},
    \[
    \sa(G)\ge \arbor(G)\ge \ceilfrac{kn}{k+n-1} =k,
    \]
    (since $k-kn/(k+n-1)<1$ for $n\ge k^2-2k+2$). Hence, $\sa(K_{k,n})=k$. 
    
    Since every induced subgraph of $K_{k,n}$ is either empty or a complete bipartite graph, we have, by \eqref{eq:parden2},
    \[
    \parden(K_{k,n})=\max \left(\frac{\sum_{i=1}^m k_i n_i}{\max\{k_i+n_i:\, 1\le i\le m\}}\right),
    \]
    where the maximum is taken over all $m\ge 1$, $k_1,\ldots,k_m$ such that $k_1+\cdots+k_m\le k$, and $n_1,\ldots,n_m\ge 1$ such that $n_1+\cdots+n_m\le n$. Note that for every such $k_1,\ldots,k_m$ and $n_1,\ldots,n_m$, we have
    \[
       \frac{\sum_{i=1}^m k_i n_i}{\max\{k_i+n_i:\, 1\le i\le m\}}\le \sum_{i=1}^m \frac{k_i n_i}{k_i+n_i} <\sum_{i=1}^m k_i \le k. 
    \]
    Hence, $\parden(K_{k,n})<k$.
\end{proof}

We proceed to prove Theorem \ref{thm:weak_brouwer_improved}.

\begin{proof}[Proof of Theorem \ref{thm:weak_brouwer_improved}]
    By Lemma \ref{lemma:sparse_subgraph}, there is a subgraph $G'$ of $G$ such that $\eps_k(G')\ge \eps_k(G)$ and $\parden(G')<k$. By Theorem \ref{thm:sa}, $\sa(G')\le k+15 \log{k}+65$. Hence, by Lemma \ref{lemma:star_arboricity_bound},
    \[
        \eps_k(G)\le \eps_k(G')\le k\cdot \sa(G')\le k^2+15 k \log{k}+65 k.
    \]
\end{proof}

\section{Sums of Laplacian eigenvalues and the matching number}\label{sec:matchings}

In this section, we present the proofs of Theorem \ref{thm:matching_number} and Corollary \ref{cor:weak_brouwer}. 
We will need the following result, which gives an upper bound on $\eps_k(G)$ in terms of the number of vertices in the largest connected component of $G$.

\begin{proposition}\label{prop:small_n_components}
    Let $G=(V,E)$ be a graph. Then, for all $1\leq k\leq |V|$,
    \[
        \eps_k(G)\le \lfrac{k\cdot n'(G)}{2}.
    \]
\end{proposition}
\begin{proof}
    Let $n'=n'(G)$. Since $\sum_{i=1}^k \lambda_i(L(G))\le \sum_{i=1}^{|V|} \lambda_i(L(G))= 2|E|$, we have $\eps_k(G)\le |E|$. On the other hand, by Corollary \ref{cor:n'_bound}, $\eps_k(G)\le kn'-|E|$. So,
\begin{equation}\label{eq:trivial}
    \eps_k(G)\le \min\{kn'-|E|,|E|\}.
\end{equation}
We divide into two cases. If $|E|\le \lfrac{kn'}{2}$, then, by \eqref{eq:trivial},
\[
    \eps_k(G)\le |E|\le \lfrac{kn'}{2}.
\]
Otherwise, $|E|\ge \ceilfrac{kn'}{2}$, and then, by \eqref{eq:trivial},
\[
 \eps_k(G)\le kn'-|E|\le kn'-\ceilfrac{kn'}{2} = \lfrac{kn'}{2}.
\]
\end{proof}

\begin{proof}[Proof of Theorem \ref{thm:matching_number}]
    By Theorem \ref{thm:edmonds}, there exists an odd set cover $C=\{v_1,\ldots,v_r,S_1,\ldots,S_t\}$ of $G$ with $w(C)=r+\sum_{i=1}^t(|S_i|-1)/2 =\nu(G)$. 
    Let $S=\{v_1,\ldots,v_r\}$, $E_1=\{e\in E:\, e\cap S\ne \emptyset\}$, and $G_1=(V,E_1)$. By  the definition of $G_1$, we have $\tau(G_1)\le |S|= r$. Hence, by Corollary \ref{cor:covering},
    \[
    \eps_k(G_1)\le k \cdot r.
    \]
    Let $G_2=(V,E\setminus E_1)$. By the definition of an odd set cover, every edge in $G_2$ is contained in one of the sets $S_1,\ldots,S_t$. Hence, the number of vertices in the largest connected component of $G_2$ is at most 
    \[
        \max \left\{ |S_i| : \, 1\le i\le t\right\} \le 1+\sum_{i=1}^t \left(|S_i|-1\right) =  2\nu(G)-2r+1.
    \]
    Therefore, by Proposition \ref{prop:small_n_components},
    \[
        \eps_k(G_2)\le \lfrac{k(2\nu(G)-2r+1)}{2}= k\cdot  \nu(G) -k \cdot r +\lfrac{k}{2}.
    \]
    By Lemma \ref{lemma:excess}, we obtain
    \[
        \eps_k(G)\le \eps_k(G_1)+\eps_k(G_2) \le k\cdot \nu(G)+\lfrac{k}{2},
    \]
    as wanted.
\end{proof}

\begin{proof}[Proof of Corollary \ref{cor:weak_brouwer}]
    Let $G=(V,E)$ be a graph. By Corollary \ref{corollary:bounded_generalized_matching}, there exists a subgraph $G'$ of $G$ such that $\eps_k(G')\ge \eps_k(G)$ and $\nu(G')\le 2k-1$.  
    By Theorem \ref{thm:matching_number}, we obtain
    \[
        \eps_k(G)\le \eps_k(G')\le k\cdot \nu(G') + \lfrac{k}{2} \le 2k^2-k+\lfrac{k}{2}=2k^2-\ceilfrac{k}{2}.
    \]
\end{proof}
\begin{remark}
For a bipartite graph $G$, K\"onig's theorem (see, for example, \cite[Theorem 1.1.1]{lovasz1986book}) states that $\tau(G)= \nu(G)$. Therefore, by Corollary \ref{cor:covering}, we have $\eps_k(G)\le k\cdot \nu(G)$. Applying Corollary \ref{corollary:bounded_generalized_matching} as in the proof of Corollary \ref{cor:weak_brouwer}, we obtain that for every bipartite graph $G=(V,E)$ and every $1\le k\le |V|$, $\eps_k(G)\le 2k^2-k$, slightly improving upon the bound in Corollary \ref{cor:weak_brouwer}.

\end{remark}

\section{Concluding remarks}\label{sec:conc}

Equality in the bound in Theorem \ref{thm:matching_number} is achieved when $G$ is a complete graph on $n$ vertices for odd $n$ and $k=n-1$, and when $G$ is a star graph and $k=1$. We believe, however, that the following stronger bound may hold.

We say that a vertex $v$ in a graph $G$ is \emph{non-isolated} if $\deg(v)\ge 1$.

\begin{conjecture}\label{conj:matching_number_improved_bound}
Let $G=(V,E)$ be a graph with $n$ non-isolated vertices, and let $1\leq k\leq n-2$. Then,
\[
    \sum_{i=1}^k \lambda_i(L(G)) \le |E|+ k\cdot \nu(G).
\]
\end{conjecture}
Note that the requirement $k\le n-2$ in Conjecture \ref{conj:matching_number_improved_bound} is necessary. Indeed, as mentioned above, if $n$ is an odd integer,  $k=n-1$, and  $K_n=(V,E)$ is the complete graph on $n$ vertices, then we have $\nu(K_n)=(n-1)/2$ and $\sum_{i=1}^{k}\lambda_i(L(K_n))=2|E|=|E|+\binom{n}{2} = |E| +k\cdot \nu(G)+(n-1)/2$. 
However, since for every graph $G=(V,E)$ with $n$ non-isolated vertices,  $\sum_{i=1}^k \lambda_i(L(G))=2|E|$ for all $k\ge n-1$, the restriction $k\le n-2$ in Conjecture \ref{conj:matching_number_improved_bound} is not a significant one.

Regarding the relation between the Laplacian eigenvalues of a graph and its covering number, we propose the following conjecture. 

\begin{conjecture}\label{conj:cover}
    Let $G=(V,E)$ be a graph. Then, for all $\tau(G)\le k\le |V|$,
    \begin{equation}\label{eq:cover_conj}
        \sum_{i=1}^k \lambda_i (L(G)) \le 
        |E|+ k\cdot \tau(G) - \binom{\tau(G)}{2}.
    \end{equation}
\end{conjecture}
For $1\le r\le n$, let $S_{n,r}$ be the graph on vertex set $[n]$ whose edges are all the pairs $\{i,j\}$ where $1\le i \le r$ and $i<j\le n$. It is not hard to check that the graph $S_{n,r}$ attains equality in \eqref{eq:cover_conj} for all $r\le k\le n-1$.

Note that Conjecture \ref{conj:cover} would improve on Corollary \ref{cor:covering} for $k\ge \tau(G)$. Moreover, the bound in Conjecture \ref{conj:cover}
is stronger that the bound $|E|+\binom{k+1}{2}$ from Conjecture \ref{con:brouwer} for all $k\ge \tau(G)$ (since the function $kx-\binom{x}{2}$ is increasing for $x\le k$).

In light of Theorem \ref{thm:weak_brouwer_improved}, we suggest the following problem as a potential step towards Brouwer's conjecture.

\begin{problem}
    Find the smallest possible constant $1/2\le C\le 1$ for which the following holds: For every $\varepsilon>0$ there exists $k_0\ge 1$ such that for every $k\ge k_0$ and every graph $G=(V,E)$ with $|V|\ge k$,
   \[
        \sum_{i=1}^k \lambda_i(L(G)) \le |E| + (1+\varepsilon)C k^2.
    \]
\end{problem}
Theorem \ref{thm:weak_brouwer_improved} shows that this statement holds for $C=1$. Proving the statement for $C=1/2$ would provide an ``approximate" version of Conjecture \ref{con:brouwer}.

\bibliographystyle{abbrv}
\bibliography{main}

\end{document}